\documentclass{article}

\usepackage{amsthm,amsmath}
\usepackage{amssymb}
\usepackage{a4wide}

\theoremstyle{plain}
\newtheorem{theorem}{Theorem}[section]

\newtheorem{proposition}[theorem]{Proposition}
\newtheorem{corollary}[theorem]{Corollary}

\theoremstyle{definition}

\theoremstyle{remark}
\newtheorem{remark}[theorem]{Remark}

\newcommand{\Spin}{\mathrm{Spin}}

\providecommand{\keywords}[1]
{	{ \small
  \textbf{\textit{Keywords---}} #1
}
}

\begin{document}

\title{Higher homotopy groups of topological Kac--Moody groups, Whitehead tower, and string groups}

\author{Ralf K\"ohl}

\maketitle

\begin{abstract}
\noindent
This note establishes that homotopy groups of topological split real Kac--Moody groups are countable and, hence, concludes the existence of Whitehead towers consisting of topological groups for these groups and their maximal compact subgroups. Moreover, this note proposes a construction for string groups for the $E_n$-series.
\end{abstract}


\section{Introduction}

Using the fiber bundle 
\begin{eqnarray*}
\mathrm{O}(n-1) \to \mathrm{O}(n) \to \mathrm{O}(n)/\mathrm{O}(n-1) \cong \mathbb{S}^{n-1} \label{classicalO}
\end{eqnarray*}
one observes that the homotopy groups $\pi_k(\mathrm{O}(n))$ are independent of $n$ whenever $n>k+1$. For the resulting {\em stable orthogonal group} $$\mathrm{O}(1) \to \mathrm{O}(2) \to \mathrm{O}(3) \to \mathrm{O}(4) \to \mathrm{O}(5) \to \cdots \to \mathrm{O} := \bigcup_{n \geq 1} \mathrm{O}(n)$$ the Bott periodicity theorem \cite{Bott:1959} states 
$$\pi_k(\mathrm{O}) = \left\{ \begin{array}{cc}  C_2 & \text{for $k=0,1 \mod 8$}, \\ \mathbb{Z} & \text{for $k=3,7 \mod 8$}, \\ \{ 1 \} & \text{else}.  \end{array} \right.  $$
Specializing to $\mathrm{O}(16)$, the first difference is in homotopy dimension $15$; indeed $\pi_{15}(\mathrm{O}(16)) = \mathbb{Z} \oplus \mathbb{Z} \neq \mathbb{Z} = \pi_{15}(\mathrm{O})$; the {\em Whitehead tower} for $\mathrm{O}(16)$ consists of homomorphisms of topological groups where in each step the lowest non-trivial homotopy is killed $$ \cdots \to \mathrm{String}(16) \to \mathrm{Spin}(16) \to \mathrm{SO}(16) \to \mathrm{O}(16)$$ and $\mathrm{String}(16)$ can be modelled as a topological group via the construction used for proving \cite[Theorem~5.1]{Stolz:1996}; or, alternatively, via \cite[Theorem 1.2]{BaezStevensonCransSchreiber:2007}; or \cite[Section~5]{Stolz/Teichner:2004}. 

On a more abstract level, any Lie group admits a Whitehead tower consisting of topological groups, by \cite[Theorem~2.13]{Rovelli:2019}, since their homotopy groups are countable, cf. \cite{Milnor:1959} for the fact and \cite[Proposition~1.14]{Rovelli:2019} for its relevance.

\medskip
The group $\mathrm{SO}(n)$ is the maximal compact subgroup of the semisimple split real Lie group $\mathrm{SL}(n,\mathbb{R})$, being the fixed-point group of a Cartan involution. Given an algebraically simply-connected semisimple split real Kac--Moody group $G$ of simply-laced type, as in \cite[Definition~1.1]{Harring/Koehl:2023}, the purpose of this note is to study the higher homotopy groups of the fixed-point subgroup $K$ of the Cartan--Chevalley involution of $G$, as in \cite[Notation~2.2]{Harring/Koehl:2023}. Moreover, this note will point out that $K$ also admits a Whitehead tower, and in fact one in which the spin cover of $K$ constructed in \cite{Ghatei/Horn/Koehl/Weiss:2017} fits. Finally, this note proposes a string group for $K$ inspired by \cite[Theorem~5.1]{Stolz:1996}, that also fits into the Whitehead tower.  

 In \cite{Hartnick/Koehl/Mars:2012} the authors constructed a topology $\tau$ on the Kac--Moody group $G$, the {\em Kac--Petersen topology}, such that the following hold:
\begin{enumerate}
\item[(i)] $\tau$ is Hausdorff and $k_\omega$ (\cite[Proposition~7.10]{Hartnick/Koehl/Mars:2012}), in particular paracompact (\cite[(6), p.\ 114]{Franklin/Thomas:1977})
\item[(ii)] $\tau$ is a group topology (\cite[Proposition~7.10]{Hartnick/Koehl/Mars:2012}),
\item[(iii)] $\tau$ induces the Lie group topology on any Levi factor of a spherical parabolic subgroup (\cite[Corollary~7.16(iv)]{Hartnick/Koehl/Mars:2012}, \cite[Proposition~3.4.8]{Harring:2020}),
\item[(iv)] $\tau$ is the finest topology satisfying these properties (\cite[Proposition~7.21]{Hartnick/Koehl/Mars:2012}).
\end{enumerate} 

By \cite[Proposition~4.9]{Harring/Koehl:2023}, in case the Kac--Moody group $G$ is of simply-laced type distinct from $A_1$, the group $(K,\tau_{|K})$ (and hence the group $G$ by \cite[Theorem~A.15]{Harring/Koehl:2023}) admits a topologically simply connected two-fold spin cover.

As mentioned above, the purpose of this note is to determine the higher homtopy groups of $K$ and of $G$.

\bigskip \noindent {\bf Acknowledgement.} The author expresses his gratitude to Karl-Hermann Neeb for a discussion concerning fibrations.

\section{Homogeneous spaces of $k_\omega$-Kac--Moody groups}

The key point of this section is that $k_\omega$-spaces are paracompact (\cite[(6), p.\ 114]{Franklin/Thomas:1977}) making the Huebsch--Hurewicz theorem \cite{Huebsch:1955} applicable.

\begin{proposition} \label{hurewicz}
Let $G$ be a split real Kac--Moody group endowed with its $k_\omega$-group topology and let $H$ be a Zariski-closed subgroup of $G$. Then $H \to G \to G/H$ is a Serre fibration and, in fact, a Hurewicz fibration.
\end{proposition}

\begin{proof}
Let $G = \bigcup_n G_n$ be a $k_\omega$-decomposition and let $q: G \to G/H$ be the quotient map. Then $G/H = \bigcup_n q(G_n)$ is a $k_\omega$-decomposition by \cite[(11), p.\ 116]{Franklin/Thomas:1977} and, in particular, the Hausdorff space $G/H$ is $k_\omega$ as well. The spaces $q(G_n)$ being compact for each $n$, by the Steenrod--Hurewicz theorem \cite[Theorem~11.3]{Steenrod:1951} the maps $G_n \to q(G_n)$ are Hurewicz fibrations and, thus, have the homotopy lifting property. Homotopies of disks (within $G/H$) are compact, hence contained in one of the $q(G_n)$ and so lift to $G_n$. Consequently, $H \to G \to G/H$ is a Serre fibration.

For the stronger claim note that $k_\omega$-spaces are paracompact by  \cite[(6), p.\ 114]{Franklin/Thomas:1977}. Therefore the Huebsch--Hurewicz theorem \cite{Huebsch:1955} implies that $q : G \to G/H$ is a Hurewicz fibration.
\end{proof}

A recurring theme in this note will be fibrations of the form $$K_J \to K \to K/K_J$$ where $K$ is the (closed) subgroup consisting of the fixed-points of the (continuous) Cartan--Chevalley involution of an algebraically simply connected split real Kac--Moody group $G$ of irreducible simply-laced type distinct from $A_1$ endowed with the subspace topology induced by the Kac--Peterson topology and $K_J := K \cap P_J$ for a $J$-parabolic subgroup $P_J$ of $G$.

Note that $$K/K_J \to G/P_J$$ is the universal covering map of the generalized flag manifold $G/P_J$ by \cite[Lemma 4.8]{Harring/Koehl:2023}. The topological structure of Kac--Moody generalized flag manifolds can be described as follows:

\begin{proposition}[{\cite[Proposition~3.7]{Harring/Koehl:2023}}] \label{flagmanifold}
Let $G$ be a split real Kac--Moody group of symmetrizable type. For each subset $J$ of the set of fundamental roots, the Bruhat decomposition $$G/P_J = \bigsqcup_{w \in W^J} BwP_J/P_J$$ is a CW decomposition, where $P_J$ is a parabolic subgroup of type $J$, the set $W^J$ equals the set of minimal coset representatives of $W/W_J$, and the dimension of each cell is given by the word length of $w \in W^J$ with respect to the set $S$ of fundamental generators of $W$.  
\end{proposition}

\begin{corollary} \label{CW}
Let $G$ be a split real Kac--Moody group of irreducible symmetrizable type, let $K$ be the fixed-point subgroup of $G$ with respect to the Cartan--Chevalley involution, and let $\mathrm{Spin} \to K$ be the double cover constructed in \cite{Ghatei/Horn/Koehl/Weiss:2017}, all groups endowed with the Kac--Peterson topology (and its obvious derivates). Then $K$ and $\mathrm{Spin}$ both admit a CW decomposition.
\end{corollary}

\begin{proof}
By Proposition~\ref{flagmanifold} the building admits a CW decomposition $$K/(T \cap K) \cong G/P_\emptyset = G/B = \bigsqcup_{w \in W} BwB/B,$$ where the homeomorphism $K/(T \cap K) \cong G/B$ follows from \cite[Lemma~4.1]{Harring/Koehl:2023}. Since $T \cap K$ is finite (\cite[Lemma~3.26]{Freyn/Hartnick/Horn/Koehl:2020}), the surjections $\mathrm{Spin} \to K \to K/(T \cap K)$ in fact are coverings onto a CW complex, yielding CW decompositions of $K$ and $\mathrm{Spin}$.
\end{proof}

\section{Homotopy groups of Kac--Moody groups}

\begin{theorem}
Simply-laced Kac--Moody groups and their maximal compact subgroups have countable homotopy groups.
\end{theorem}

\begin{proof}
Property (iii) of the Kac--Peterson topology $\tau$ from the introduction implies that any embedded subgroup $\mathrm{SO}(3) \stackrel{\iota}{\longrightarrow} K$ along an $A_2$-subdiagram inherits the Lie topology when restricting $\tau$. Palais' slice theorem \cite{Palais:1961} becomes applicable and yields the homotopy exact sequence
\begin{eqnarray}
\pi_{k+1}(K/\mathrm{SO}(3)) \longrightarrow  \pi_k(\mathrm{SO}(3)) \stackrel{\pi_k(\iota)}{\longrightarrow} \pi_k(K) \stackrel{\pi_k(\pi)}{\longrightarrow}
\pi_k(K/\mathrm{SO}(3)) \label{countable}
\end{eqnarray}
for any such fibering $\mathrm{SO}(3) \stackrel{\iota}{\longrightarrow} K \stackrel{\pi}{\longrightarrow} K/\mathrm{SO}(3)$ along an $A_2$-subdiagram.

The map $K/\mathrm{SO}(3) \to K/ \mathrm{SO}(3)(T \cap K)$ affords a finite covering (see \cite[Lemma~4.3]{Harring/Koehl:2023}). In particular, the Bruhat decomposition of $K/\mathrm{SO}(3)(T \cap K) \cong G/P$ (cf.\ \cite[Lemma~4.1]{Harring/Koehl:2023} for this homeomorphism), where $P$ is a parabolic corresponding to the embedded $A_2$-diagram from above, induces a countable CW complex structure (see \cite[Theorem~6.56, Remark~1]{Abramenko/Brown:2008} and \cite[Proposition~3.7]{Harring/Koehl:2023}; see also Proposition \ref{flagmanifold} above), which lifts to a countable CW complex structure on $K/\mathrm{SO}(3)$. So, all $\pi_k(K/\mathrm{SO}(3))$, $k \geq 0$, are countable.

The Lie group $\mathrm{SO}(3)$ (see property (iii) of the Kac--Peterson topology $\tau$ from the introduction) also has countable $\pi_k(\mathrm{SO}(3))$, $k \geq 0$ (see \cite{Milnor:1959}).

Therefore the homotopy exact sequence (\ref{countable}) implies $\pi_k(K)$, $k \geq 0$, countable by the homomorphism theorem of groups: the factor group $\pi_k(K)/\pi_k(\iota)(\pi_k(\mathrm{SO}(3)))$ is isomorphic to a subgroup of $\pi_k(K/\mathrm{SO}(3))$. 
\end{proof}

\begin{remark}
\begin{enumerate}
\item By \cite[Corollary~A.12]{Harring/Koehl:2023} the preceding proof establishes countable $\pi_k(G) \cong \pi_k(K)$, $k \geq 0$, as well.
\item The preceding proof in fact works verbatim for any irreducible type distinct from $A_1$ in which the non-zero off-diagonal entries of the generalized Cartan matrix are all odd (see \cite[Notation~3.13]{Harring/Koehl:2023}).
\item Countability of $\pi_k(K)$, $k \geq 0$, in fact holds for any symmetrizable type, but one cannot expect $\Spin \to K$ to be universal in this generality (see \cite[Third Theorem]{Harring/Koehl:2023}).
\end{enumerate}
\end{remark}

Computing the homotopy groups of the maximal compact group $K = K(E_n)$ in the $E_{n}$ series is of particular interest. Conveniently, for $E_8$ one has $K \cong \mathrm{SO}(16)$, so the first $15$ homotopy groups are listed in the introduction.

\begin{theorem}
Let $n \geq 8$ and let $K(E_n)$ be the maximal compact subgroup of type $E_n$. Then for $1 \leq k \leq 6$ one has $$\pi_k(K(E_n)) = \pi_k(\mathrm{SO}(16)).$$
In particular, $\pi_1(K(E_{10})) = C_2$, $\pi_3(K(E_{10})) = \mathbb{Z}$ and $\pi_0(K(E_{10})) = \pi_2(K(E_{10})) = \pi_4(K(E_{10})) = \pi_5(K(E_{10})) = \pi_6(K(E_{10})) = \{ 1 \}$.
\end{theorem}

\begin{proof}
One has $\mathrm{SO}(9)/\mathrm{SO}(8) \cong \mathbb{S}^8$. Considering $\mathrm{SO}(n)$ as the maximal compact subgroup of $\mathrm{SL}(n,\mathbb{R})$, one obtains $K(A_8)/K(A_7) \cong \mathbb{S}^8$. Comparing low-dimensional cells as in \cite[p.~173]{Kramer:2002} yields identical cell decompositions up to dimension $7$  in $K(E_9)/K(E_8)$ (in the sense of \cite[Proposition~3.7]{Harring/Koehl:2023}; see also Proposition \ref{flagmanifold} above): only beyond the coset $s_3s_4s_5s_6s_7s_8s_9W_{E_8}$ inside $W_{E_9}$ one can actually notice the differences between the diagrams $A_8$ and $E_9$ (see \cite[Lemma~2.4.3]{Bjoerner/Brenti:2005}). Hence for all $1 \leq k \leq 7$ one has $$\pi_k(K(E_9)/K(E_8)) = \pi_k(\mathbb{S}^8) = \{ 1 \}.$$
Proposition~\ref{hurewicz} yields the homotopy exact sequence 
\begin{eqnarray*}
\{ 1 \} = \pi_{k+1}(K(E_9)/K(E_8)) \to  \pi_k(K(E_8)) \to \pi_k(K(E_9)) \to
\pi_k(K(E_9)/K(E_8)) = \{ 1 \}, \quad\text{$1 \leq k \leq 6$,}
\end{eqnarray*}
from which the assertion of the theorem follows for $n=9$. A straightforward induction on $n$ finishes the proof.
\end{proof}

A more detailled investigation requires a better understanding of the set $W^J$ from \cite[Proposition~3.7]{Harring/Koehl:2023} (see also Proposition~\ref{flagmanifold} above). There exist standard methods of investigation in Coxeter group combinatorics \cite[Chapter 7]{Bjoerner/Brenti:2005}.

\section{String groups}
\begin{theorem} \label{string}
For $n \geq 8$ there exists a topological group $\mathrm{String}(E_n)$ with vanishing homotopy up to and including dimension $6$ and homotopy identical to $\mathrm{Spin}(E_n)$ beyond that admitting a quotient map $\mathrm{String}(E_n) \to \mathrm{Spin}(E_n)$.
\end{theorem}

The proof follows the strategy of \cite[Proof of Theorem~5.1]{Stolz:1996}, reproduced here for the reader's convenience as a series of the following results:

\begin{proposition} \label{projectiveunitary}
Let $H$ be an infinite-dimensional separable complex Hilbert space and define $\mathrm{PU}(H) := \mathrm{U}(H)/ \{ \lambda \cdot \mathrm{id} \mid \lambda \in \mathrm{U}_1(\mathbb{C}) \}$ with respect to the norm topology on the unitary group $\mathrm{U}(H)$. Then $\mathrm{U}_1(\mathbb{C}) \to \mathrm{U}(H) \to \mathrm{PU}(H)$ is a Hurewicz fibration, and the homotopy of $\mathrm{PU}(H)$ is trivial with the exception of $\pi_2(\mathrm{PU}(H)) = \mathbb{Z}$.
\end{proposition}

\begin{proof}
The norm topology turns $\mathrm{U}(H)$ into a contractible group (since $H$ is infinite-dimensional). The embedded circle group $\mathrm{U}_1(\mathbb{C})$ carries the Lie topology, and hence has trivial homotopy with the exception of $\pi_1(\mathrm{U}_1(\mathbb{C})) = \mathbb{Z}$. By Palais' slice theorem \cite{Palais:1961} the fiber bundle $\mathrm{U}_1(\mathbb{C}) \to \mathrm{U}(H) \to \mathrm{PU}(H)$ is locally trivial and, hence, a Hurewicz fibration. The resulting homotopy exact sequence $$\{ 1 \} = \pi_2(\mathrm{U}(H)) \to \pi_2(\mathrm{PU}(H)) \to \pi_1(\mathrm{U}_1(\mathbb{C})) \to \pi_1(\mathrm{U}(H)) = \{ 1 \}$$ completes the proof.
\end{proof}

\begin{corollary}
Let $\mathrm{EPU}(H) \to \mathrm{BPU}(H)$ be a universal $\mathrm{PU}(H)$-bundle (in the sense of, e.g., \cite{Milnor:1956}, \cite[Theorem~4.11.2]{Husemoller:1993}). Then $\mathrm{BPU}(H)$ is an Eilenberg--MacLane space $\mathrm{K}(\mathbb{Z},3)$.
\end{corollary}

\begin{proof}
The locally trivial fiber bundle $\mathrm{PU}(H) \to \mathrm{EPU}(H) \to \mathrm{BPU}(H)$ induces the homotopy exact sequence $$\{ 1 \} = \pi_3(\mathrm{EPU}(H)) \to \pi_3(\mathrm{BPU}(H)) \to \pi_2(\mathrm{PU}(H)) \to \pi_2(\mathrm{EPU}(H)) = \{ 1 \},$$ and the claim follows from Proposition~\ref{projectiveunitary}.
\end{proof}

\begin{corollary}
Let $G(E_n)$ be an algebraically simply-connected semisimple split real Kac--Moody group of type $E_n$, $n \geq 8$, endowed with the Kac--Peterson topology $\tau$, let $K(E_n)$ be the subgroup consisting of the fixed-points of the Cartan--Chevalley involution of $G(E_n)$ endowed with the subspace topology, and let $\mathrm{Spin}(E_n) \to K(E_n)$ be the universal covering map from \cite{Harring/Koehl:2023}. Then the isomorphism classes of principal $\mathrm{PU}(H)$-bundles over $\mathrm{Spin}(E_n)$ are in one-to-one correspondence to the elements of $H^3(\mathrm{Spin}(E_n), \mathbb{Z})$ (the so-called characteristic classes).
\end{corollary}

\begin{proof}
This follows from the classification theorem of principal bundles (see \cite[Theorem~14.4.1]{tomDieck:2008}, also \cite[Theorem~4.13.1]{Husemoller:1993}) combined with the interplay between Eilenberg--MacLane spaces and cohomology (see \cite[Theorem~17.5.1]{tomDieck:2008}). Note here that $\mathrm{Spin}(E_n)$ is paracompact (finite covers of $k_\omega$-spaces are $k_\omega$) and admits a CW decomposition by Corollary~\ref{CW}; moreover, $\mathrm{BPU}(H)$ is an Eilenberg--Maclane space $\mathrm{K}(\mathbb{Z},3)$ by the preceding corollary.
\end{proof}

\begin{proof}[Proof of Theorem~\ref{string}.]
We follow \cite[Proof of Theorem~5.1]{Stolz:1996}. Consider the principal $\mathrm{PU}(H)$-bundle $P \to \mathrm{Spin}(E_n)$ corresponding to a generator of $H^3(\mathrm{Spin}(E_n), \mathbb{Z}) \cong \mathbb{Z}$. Moreover, let $\mathrm{Aut}(P)$ be the group of $\mathrm{PU}(H)$-equivariant homeomorphisms $P \to P$. Each equivariant homeomorphism $f$ of $P$ induces a homeomorphism $\tilde f$ of the base space $\mathrm{Spin}(E_n)$, yielding a continuous group homomorphism $\pi : \mathrm{Aut}(P) \to \mathrm{Homeo}(\mathrm{Spin}(E_n))$.

Define $$\mathrm{String}(E_n) := \{ f \in \mathrm{Aut}(P) : \pi(f) \in  \mathrm{Spin}(E_n) \subset \mathrm{Homeo}(\mathrm{Spin}(E_n)) \},$$ considering the embedding $\mathrm{Spin}(E_n) \subset \mathrm{Homeo}(\mathrm{Spin}(E_n))$ given by right multiplication $g \mapsto \{ x \mapsto xg \}$.

This yields the gauge bundle $$\{ 1 \} \to \mathrm{Gauge}(P) \to \mathrm{String}(E_n) \to \mathrm{Spin}(E_n)  \to \{ 1 \},$$ killing $\pi_3$ by \cite[Lemma~5.6]{Stolz:1996}.
\end{proof}


\section{Whitehead towers}

\begin{theorem}
Let $G$ be an algebraically simply-connected semisimple split real Kac--Moody group $G$ of irreducible simply-laced type distinct from $A_1$ endowed with the Kac--Peterson topology $\tau$, and let $K$ be the subgroup consisting of the fixed-points of the Cartan--Chevalley involution of $G$ endowed with the subspace topology.

Then $K$ is connected with respect to $\tau$ and admits a Whitehead tower, and one can in fact choose $$\cdots \to \Spin \to K,$$ where $\Spin$ is the two-fold spin cover constructed in \cite{Ghatei/Horn/Koehl/Weiss:2017}. If $G$ is of type $E_n$, $n \geq 8$, then one can moreover choose the Whitehead tower as $$\cdots \to \mathrm{String} \to \Spin \to K.$$ 
\end{theorem}

\begin{proof}
By \cite[Proposition~4.9]{Harring/Koehl:2023} the two-fold cover $$\Spin \to K$$ is universal with respect to $\tau$. Hence by 
\cite[Proposition~1.14]{Rovelli:2019} and \cite[Theorem~2.13]{Rovelli:2019} it suffices to prove that $K$ has countable homotopy groups. In case of type $E_n$, additionally Theorem~\ref{string} applies.
\end{proof}

\bigskip \noindent
Author's address

\medskip \noindent
Ralf K\"ohl \\
Christian-Albrechts-Universit\"at zu Kiel \\ Mathematisches Seminar \\ Heinrich-Hecht-Platz 6 \\ 24118 Kiel \\ Germany \\
{\tt koehl@math.uni-kiel.de}

\end{document}